\documentclass[sort&compress,preprint,11pt]{elsarticle}

\usepackage{amssymb}
\usepackage{amsmath,amsthm}
\usepackage{lineno,hyperref}
\usepackage{amsfonts}
\usepackage{amsmath}
\usepackage{rotating}
\usepackage{amssymb}
\usepackage{amscd,multirow}
\usepackage{graphicx}
\usepackage{epstopdf}
\usepackage[titletoc]{appendix}
\usepackage{subfigure}
\usepackage[justification=centering]{caption}
\usepackage{subeqnarray,cases}
\usepackage{braket,amsfonts}
\usepackage{cleveref}
\usepackage{mathtools}
\usepackage{array}
\usepackage{algpseudocode}%
\usepackage{listings}%
\usepackage[misc]{ifsym}
\usepackage[justification=centering]{caption}
\usepackage{subfigure,ragged2e}
\usepackage{multirow}%

\usepackage{caption}
\usepackage{subcaption}
\usepackage{subcaption,natbib}
\usepackage{pgfplots}
 
\usepackage{makecell}
\usepackage{algorithm}

\newtheorem{theorem}{Theorem}[section]

\newtheorem{lemma}{Lemma}[section]
\newtheorem{proposition}{Proposition}[section]
\newtheorem{remark}{Remark}[section]

\newtheorem{assumption}{Assumption}[section]
\allowdisplaybreaks[4]

\modulolinenumbers[5]

\usepackage[a4paper,margin=1in]{geometry}
\usepackage{amsmath,amssymb,amsthm,mathtools}
\usepackage{bm}
\usepackage{enumitem}

\newcommand{\R}{\mathbb{R}}
\newcommand{\ip}[2]{\left\langle #1,#2\right\rangle}
\newcommand{\norm}[1]{\left\|#1\right\|}

\newcommand{\La}{\mathcal{L}}
\newcommand{\E}{\mathcal{E}}

\newcommand{\bigO}{\mathcal{O}}

      \renewcommand{\thesection}{\arabic{section}}

\begin{document}

\begin{frontmatter}



\title{Trajectory convergence and  $o(t^{-2})$ rates for Nesterov accelerated primal-dual dynamics without  Lipschitz gradient assumption}


\author{Xin He} 

\affiliation{organization={School of Science, Xihua University},
            city={Chengdu},
            postcode={610039}, 
            state={Sichuan},
            country={China}}

\ead{hexinuser@163.com}

\author{Nan-Jing Huang}
\affiliation{organization={Department of Mathematics, Sichuan University},
            city={Chengdu},
            postcode={610064}, 
            state={Sichuan},
            country={China}}

\ead{njhuang@scu.edu.cn}

\author{Yi-Bin Xiao}
\affiliation{organization={Department of Mathematics, University of Electronic Science and Technology of China},
            city={Chengdu},
            postcode={611731}, 
            state={Sichuan},
            country={China}}
\ead{xiaoyb9999@hotmail.com}

\author{Ya-Ping Fang}
\affiliation{organization={Department of Mathematics, Sichuan University},
            city={Chengdu},
            postcode={610064}, 
            state={Sichuan},
            country={China}}
\ead{ypfang@scu.edu.cn}

%
%

\begin{abstract}
We consider the  Nesterov accelerated  primal-dual dynamical system
\[
\begin{cases}
\ddot{x}(t)+\dfrac{\alpha}{t}\dot{x}(t)
+\nabla f(x(t))
+A^\top\bigl(\lambda(t)+\theta t\dot{\lambda}(t)\bigr)+\beta A^\top(Ax(t)-b)=0,\\[0.6em]
\ddot{\lambda}(t)+\dfrac{\alpha}{t}\dot{\lambda}(t)
-\bigl(A(x(t)+\theta t\dot{x}(t))-b\bigr)=0,
\end{cases}
\]
which is linked to the linearly constrained optimization problem
$
    \min_{x\in\R^n} f(x),\  \text{\rm s.t. } Ax=b,
$
where  $\alpha\ge 3$ and $f$ is convex and continuously differentiable. In a Hilbert framework, the weak convergence of its trajectory was established by Bo\c{t} and Nguyen (J. Differential Equations, 303:369--406, 2021) under  $\alpha>3$ and the Lipschitz continuity assumption on $\nabla f$. In this paper, we prove in finite-dimensional spaces that the trajectory converges to a primal-dual solution for $\alpha\ge3$, without
assuming Lipschitz continuity of $\nabla f$. Moreover, when $\alpha>3$,
we establish improved $o(t^{-2})$ convergence rates for both the objective residual and
the feasibility violation. Our analysis relies on Bregman-distance arguments, instead of  the  Lipschitz continuity of  $\nabla f$.  The same strategy can also be extended
to time-scaled primal-dual dynamics to obtain analogous convergence results.
To the best of our knowledge, this is the first results in this topic without  Lipschitz gradient assumption. Our result also present the first work on the convergence of the trajectory of the accelerated primal-dual dynamical system for the critical case $\alpha=3$. 
\end{abstract}


\begin{keyword}
Linearly constrained convex optimization, Nesterov accelerated primal-dual dynamic, vanishing damping $\frac{\alpha}{t}$, trajectory convergence, $o(\frac{1}{t^{2}})$ convergence rate

\vspace{0.5em}


\end{keyword}

\end{frontmatter}

\section{Introduction}

\subsection{Problem setting}

Let $f:\R^n\to\R$ be a convex continuously differentiable function,
$A\in\R^{m\times n}$, and $b\in\R^m$. We consider the following linearly
constrained convex optimization problem in the finite-dimensional setting:
\begin{equation}\label{main_ques}
    \min_{x\in\R^n} f(x),\qquad \text{s.t. } Ax=b.
\end{equation}
Problem \eqref{main_ques} arises naturally in constrained
optimization, signal processing, machine learning, distributed optimization,
and many other areas where objective minimization is coupled with linear
conservation laws; see \cite{LinBook,BaiJSC2025,BaiJsc,Boyd}.

Recall the augmented Lagrangian function associated with \eqref{main_ques}:
\begin{equation*}\label{eq_lagrangian}
    \La_\beta(x,\lambda)
    =  f(x)+\langle \lambda,Ax-b\rangle+\frac{\beta}{2}\|Ax-b\|^2,
    \qquad \beta\ge0.
\end{equation*}
The quadratic penalty term in $\La_\beta$ provides direct control of the
feasibility residual. A pair $(x^*,\lambda^*)\in \R^n\times \R^m$
is called a saddle point of $\La_\beta$ if
\[
    \La_\beta(x^*,\lambda)\le \La_\beta(x^*,\lambda^*)\le \La_\beta(x,\lambda^*),
    \qquad
    \forall (x,\lambda)\in\R^n\times\R^m .
\]
By the optimality condition for \eqref{main_ques}, $(x^*,\lambda^*)$ also
satisfies
\begin{equation*}\label{eq:KKT}
    Ax^*=b,\qquad
    \nabla f(x^*)+A^\top\lambda^*=0
\end{equation*}
and it is also called a primal-dual solution of problem \eqref{main_ques}.

Throughout the paper, we denote the set of primal optimal solutions by
\begin{equation*}
    \mathcal S
    :=\{x^*\in\R^n\ |\ x^*\in 
    \operatorname*{argmin}\{f(x): Ax=b\}\}.
\end{equation*}
We also denote the primal-dual solution set by
\begin{equation*}\label{eq:saddle-set}
    \Omega
    :=
    \left\{(x^*,\lambda^*)\in\R^n\times\R^m:
    Ax^*=b,\ \nabla f(x^*)+A^\top\lambda^*=0\right\}.
\end{equation*}
We assume that $\mathcal S\neq\emptyset$, and so  $\Omega\ne\emptyset$ since  for every
$x^*\in \mathcal S$, there exists $\lambda^*\in\R^m$ such that
$(x^*,\lambda^*)\in \Omega$.  In general, the primal-dual solution set $\Omega$ need
not be a singleton. Nevertheless, if $(x_1^*,\lambda_1^*)$ and
$(x_2^*,\lambda_2^*)$ belong to $\Omega$, then
\[
    \nabla f(x_1^*)=\nabla f(x_2^*),
    \qquad
    A^\top\lambda_1^*=A^\top\lambda_2^*.
\]
See Proposition \ref{prop_App1} in Appendix \ref{AppA}. This elementary structural
property will be useful in the trajectory convergence analysis.

 \subsection{Inertial gradient dynamics with vanishing damping}

Continuous-time dynamical systems provide an effective viewpoint for
understanding acceleration in convex optimization. For  the unconstrained convex
minimization
\[
    \min_{x\in\mathbb R^n} f(x),
\]
one of the most important associated  models is the inertial gradient dynamic with
asymptotically vanishing damping
\[
  \text{\rm (AVD)}_{\alpha}\qquad
  \ddot x(t)+\frac{\alpha}{t}\dot x(t)+\nabla f(x(t))=0.
\]
Su, Boyd, and Cand\`es \cite{SuJMLR} showed that
$\text{\rm (AVD)}_{\alpha}$ provides a continuous-time interpretation of
Nesterov acceleration \cite{NesterovSMD}. In particular,  this dynamic yields the continuous-time counterpart of the
classical accelerated rate
$
    f(x(t))-\min f=\bigO\left({t^{-2}}\right)
$ when $\alpha\ge 3$. May \cite{MayTJM} further showed that the
$\bigO(t^{-2})$ estimate for the objective residual can be improved to
$o(t^{-2})$ when $\alpha>3$. Concerning trajectory convergence, Attouch et al.
\cite{AttouchMP} established the weak convergence of the trajectory in a Hilbert
space for the noncritical regime $\alpha>3$.  For  the critical case $\alpha=3$, the question of the weak  convergence of the trajectory to $\text{\rm (AVD)}_{\alpha}$ in a  Hilbert space remains open in a long time until Jang and Ryu \cite{Jang25} proved the global trajectory convergence
in a finite-dimensional space setting. For the discrete algorithmic counterparts of
$\text{\rm (AVD)}_{\alpha}$ (the  Nesterov accelerated algorithm), the weak convergence of the iterates can be found in Attouch et al. \cite{AttouchMP} for the noncritical case $\alpha>3$   and in Bo\c{t} et al. \cite{BotFNArx}, Jang and Ryu \cite{Jang25}  for the critical case $\alpha=3$. Further developments of
$\text{\rm (AVD)}_{\alpha}$ include time-rescaling techniques,
Hessian-driven damping, Tikhonov regularization and time discretization, see \cite{BotFocm,WibisonoPNAS,LuoMp,BotTik,AttouchJEMS,Attouchk2,Aujol,AlecsaSiopt,He2026coa}.

For the linearly constrained convex problem \eqref{main_ques}, the
unconstrained $\text{\rm (AVD)}_{\alpha}$ theory provides a useful intuition,
but it cannot be transferred directly to the primal-dual setting. A
primal-dual dynamic must simultaneously control the objective residual, the
feasibility violation, and the dual multiplier \cite{
GuoTnnls,LuoEsiam,BotPDJDE,ZengSicon}. This leads to coupling terms
between the primal and dual variables, which enforces the constraint $Ax=b$,
but makes the convergence analysis of the full trajectory more delicate.
Several inertial primal-dual systems with vanishing damping have been
developed for the linearly constrained  problem \eqref{main_ques}. Using the augmented Lagrangian $\La_\beta$, Zeng et al. \cite{ZengTAC}
introduced the Nesterov-type dynamic
\begin{equation}\label{dy_main}
\begin{cases}
\ddot{x}(t)+\dfrac{\alpha}{t}\dot{x}(t)
+\nabla f(x(t))
+A^\top\bigl(\lambda(t)+\theta t\dot{\lambda}(t)\bigr)
+\beta A^\top(Ax(t)-b)=0,\\[0.7em]
\ddot{\lambda}(t)+\dfrac{\alpha}{t}\dot{\lambda}(t)
-\bigl(A(x(t)+\theta t\dot{x}(t))-b\bigr)=0.
\end{cases}
\end{equation}
For $\alpha\ge 3$ and $\theta=1/2$, they proved the
$\bigO(t^{-2})$ decay rate of the augmented Lagrangian residual
$\La_{\beta}(x(t),\lambda^*)-\La_{\beta}(x^*,\lambda^*)$. This estimate
further implies $\bigO(t^{-1})$ convergence rates for the objective residual
$|f(x(t))-f(x^*)|$ and the feasibility violation $\|Ax(t)-b\|$. He et al.
\cite{HeSICON} and Attouch et al. \cite{AttouchJOTA} extended this line of
research to inertial primal-dual dynamics with more general damping and
scaling coefficients for separable convex optimization problems.  The   accelerated $\bigO(t^{-2})$ rates for  the objective residual and the feasibility violation can be found in   He et al. \cite{HeAA} and Bo\c{t} and Nguyen \cite{BotJDE}. Further developments
include inertial primal-dual dynamics with implicit and explicit
Hessian-driven damping \cite{HeOpt,LiCNSN},  mirror primal-dual
dynamics for distributed constrained optimization \cite{ZhaoJMLR},
``second-order primal'' + ``first-order dual'' dynamics \cite{HeAutomatica,SunJota,Battahi}, and time discretization of primal-dual
dynamical systems \cite{BotALM,ZhangAmo,BotArxiv,LuoMC}.  These works provide useful analytical tools
for studying accelerated primal-dual dynamics under linear constraints.

 \subsection{Main ideas and contributions}

Most of the above results of Nesterov accelerated primal-dual dynamics for \eqref{main_ques} focus primarily on convergence rates for the
objective residual and  the feasibility violation  under $\alpha\ge 3$, the existing estimates yield
$\bigO(t^{-2})$ decay rates. By contrast, in the unconstrained case, the
dynamic $\text{\rm (AVD)}_{\alpha}$ admits the improved $o(t^{-2})$ decay
rate when $\alpha>3$ (see\cite{MayTJM,BotTik}). Thus, compared with the
unconstrained theory, the existing convergence results for \eqref{dy_main}
do not fully capture the   asymptotic behavior suggested by the
vanishing damping $\alpha/t$. The trajectory convergence of primal-dual dynamics is more subtle and was
not addressed in many earlier works. A major contribution in this direction
is due to Bo\c{t} and Nguyen \cite{BotJDE} who proved the weak convergence of the trajectory of the dynamic
\eqref{dy_main} in
Hilbert spaces under the Lipschitz continuity of $\nabla f$ and the
noncritical damping condition $\alpha>3$. The trajectory convergence results were further  extended to a time-rescaled version
of \eqref{dy_main} by Hulett and Nguyen \cite{HulettAMO} under the same
Lipschitz gradient assumption and $\alpha>3$.

When $\nabla f$ is Lipschitz continuous, one can use the cocoercivity
inequality to control the gradient residual. This is a key step in Bo\c{t} and
Nguyen \cite{BotJDE} in proving
the asymptotic stationarity
\begin{equation}\label{eq_saddle}
	 \nabla f(x(t))+A^\top\lambda(t)\to0.
\end{equation}
However, in unconstrained optimization, the trajectory convergence of
$\text{\rm (AVD)}_{\alpha}$ can be obtained without assuming Lipschitz
continuity of $\nabla f$; see \cite{AttouchMP,Jang25}. This suggests that
Lipschitz smoothness may not be essential to the continuous-time acceleration
mechanism. The main difficulty is whether such a non-Lipschitz analysis can
be extended from inertial dynamics for unconstrained convex minimization to
the primal-dual dynamic \eqref{dy_main} for the linearly constrained problem \eqref{main_ques}.

The above discussion shows that several basic questions  which were not addressed  in the pioneering work of Bo\c{t} and Nguyen \cite{BotJDE}  remain open for the primal-dual dynamic \eqref{dy_main}. 
\[
\boxed{
\begin{minipage}{0.95\textwidth}
\indent\par
(i) Can the weak convergence of the trajectory  of \eqref{dy_main} in a Hilbert framework be guaranteed for the noncritical damping condition $\alpha=3$?

\indent\par (ii) Can the  $\bigO(t^{-2})$ decay  for the objective residual and the feasibility violation  be improved to
$o(t^{-2})$ for  Nesterov-type accelerated primal-dual dynamics when $\alpha>3$?

\indent\par
 (iii) Existing convergence results for inertial primal-dual dynamics in the literature always require the Lipschitz  gradient assumption. Whether do  these convergence properties still hold without Lipschitz  gradient assumption?
\end{minipage}
}
\]

In this paper we shall give  affirmative answers to these problems  in a finite-dimensional space setting. Throughout,
we assume only that $f$ is convex and continuously differentiable, and 
$\nabla f$ is not required to be Lipschitz continuous.
The convergence analysis presented in this paper relies on a Bregman-distance
argument  instead of the standard cocoercivity argument used in existing works. For a fixed primal solution $x^*\in\mathcal S$, we use
\begin{equation}\label{eq_Df}
	D_f(x^*,x)
=
f(x^*)-f(x)
-\left\langle \nabla f(x),x^*-x\right\rangle,\quad \forall x\in \R^n.
\end{equation}
By the convexity of $f$, $D_f(x^*,x)\ge0$. The Lyapunov structure of
\eqref{dy_main} yields the weighted integrability estimate
\[
    \int_{t_0}^{+\infty} tD_f(x^*,x(t))\,dt<+\infty.
\]
Together with the velocity estimate $\|\dot{x}(t)\|=\bigO(1/t)$ and the boundedness of trajectories $x(t)$ in the finite-dimensional space, this implies
$
    D_f(x^*,x(t))\to0
$
as $t\to\infty$. This is a key step in establishing \eqref{eq_saddle}
without assuming Lipschitz continuity of $\nabla f$.

The main contributions of this paper are summarized as follows.

\begin{itemize}
    \item In a finite-dimensional space setting, we prove  the trajectory convergence of the Nesterov accelerated  primal-dual dynamic
\eqref{dy_main}  to a primal-dual solution of the  problem \eqref{main_ques} under the 
assumption that $f$ is continuously differentiable and $\alpha\ge 3$. Compared with
the result of Bo\c{t} and Nguyen \cite{BotJDE}, where weak convergence was
obtained under the Lipschitz continuity of $\nabla f$ and the noncritical
condition $\alpha>3$, our result removes the Lipschitz gradient assumption
and covers the critical damping case $\alpha=3$. To the best of our
knowledge, this is the first trajectory convergence result for Nesterov accelerated primal-dual dynamical systems in the critical case $\alpha=3$. The proof does
not rely on the cocoercivity-based gradient residual estimate or Opial's
lemma. Instead, it is based on a Bregman-distance argument and on the
existence and uniqueness of cluster points of the trajectory. In the
unconstrained case, this result recovers the corresponding trajectory
convergence properties of the classical $\text{\rm (AVD)}_{\alpha}$ dynamics
\cite{AttouchMP,Jang25}.

    \item Without Lipschitz gradient assumption, we establish the improved $o(t^{-2})$ convergence rates for both the
    objective residual and the feasibility violation in the noncritical
    regime $\alpha>3$. This improves the existing $\bigO(t^{-2})$ estimates
    for Nesterov-type primal-dual dynamical systems
    \cite{ZengTAC,HeSICON,ZhaoJMLR,AttouchJOTA,BotJDE}. When $A=0$ and $b=0$, our result recovers the existing $o(t^{-2})$ convergence rates for the unconstrained $\text{\rm (AVD)}_{\alpha}$ dynamic \cite{MayTJM}. Furthermore, the  Bregman-distance strategy is adapted to time-scaled Nesterov primal-dual dynamical system considered in
    \cite{HulettAMO,HeAA}, yielding analogous trajectory convergence results
    and improved convergence rates under the corresponding parameter
    conditions.
\end{itemize}

\subsection{Organization}

The rest of the paper is organized as follows. Section~\ref{sec2}
establishes some basic estimates and proves the convergence of the trajectory
to a primal-dual solution of problem \eqref{main_ques}. Section~\ref{sec3} proves the improved $o(t^{-2})$ rates
for the objective residual and the feasibility violation. Section~\ref{sec4}
discusses extensions to time-scaled Nesterov primal-dual dynamical systems.
Section~\ref{sec5} concludes the paper. Some useful auxiliary lemmas are
provided in Appendix \ref{AppA}.

\section{Estimates and trajectory convergence}\label{sec2}

In this section, we investigate the convergence of the trajectory generated
by the dynamic \eqref{dy_main}. Throughout this section, the parameters are
assumed to satisfy the following condition.

\begin{center}
\vspace{0.5em}
\setlength{\fboxsep}{6pt}
\setlength{\fboxrule}{0.4pt}
\fbox{%
\begin{minipage}{0.96\textwidth}
\begin{assumption}\label{ass1}
The parameters $\alpha$, $\beta$, and $\theta$ in the dynamic \eqref{dy_main} satisfy
\[
\alpha\ge 3, \qquad
\theta\in\left[\frac{1}{\alpha-1},\frac12\right],\qquad
\beta\ge 0.
\]
\end{assumption}
\end{minipage}%
}
\vspace{0.5em}
\end{center}

Under Assumption \ref{ass1}, for any initial point
\[(x(t_0),\lambda(t_0),\dot x(t_0),\dot\lambda(t_0))
=(x_0,\lambda_0,\dot x_0,\dot\lambda_0),\]
 the existence and uniqueness of a
global solution to the dynamic \eqref{dy_main} has been investigated in
\cite{HeSICON,AttouchJOTA,BotJDE}. Throughout this paper, we assume that
\eqref{dy_main} admits a global solution $(x(t),\lambda(t))$.

To study trajectory convergence, we fix
$z^*=(x^*,\lambda^*)\in \Omega$ and recall the following energy function
$\E_{z^*}:[t_0,+\infty)\to [0,+\infty)$:
\begin{equation}\label{Edef}
    \E_{z^*}(t)=\theta^2t^2(\La_\beta(x(t),\lambda^*)- \La_\beta(x^*,\lambda(t)))+\frac{1}{2}\norm{v(t)}^2+\frac{\xi}{2}\norm{z(t)-z^*}^2,
\end{equation}
where $\xi:=\theta\alpha-\theta-1\ge0$ and
\[
    v(t)=z(t)-z^*+\theta t\dot z(t),
    \qquad
    z(t)=(x(t),\lambda(t)).
\]
The energy function \eqref{Edef} was used in \cite{BotJDE} for the
convergence analysis of \eqref{dy_main}. Based on this energy, we first
establish several estimates that will be used later in the trajectory
convergence analysis. Some related results can also be found in
\cite{BotJDE,HeAA,ZengTAC,HeSICON}.

\begin{theorem}\label{th_basic}
Let $(x(t),\lambda(t))$ be a trajectory of the dynamic \eqref{dy_main}, and
 $(x^*,\lambda^*)\in\Omega$. Then the following conclusions hold:
\begin{itemize}
	\item[$(i)$] Integral estimates:
	\begin{align*}
&(1-2\theta)\int_{t_0}^{+\infty} t(\La_\beta(x(t),\lambda^*)- \La_\beta(x^*,\lambda(t)))\,dt<+\infty, \\
&(\theta\alpha-\theta-1)\int_{t_0}^{+\infty} t\norm{(\dot x(t),\dot \lambda(t))}^2\,dt<+\infty, \\
&\int_{t_0}^{+\infty} tD_f(x^*,x(t))\,dt<+\infty.
\end{align*}
	\item[$(ii)$] The trajectory $(x(t),\lambda(t))$ is bounded on
    $[t_0,+\infty)$, and the velocity satisfies
\[
        \norm{(\dot x(t), \dot \lambda(t))}=\bigO\left(\frac1t\right).
\]
	\item[$(iii)$] The objective residual and the feasibility violation satisfy
	\[
        \|Ax(t)-b\|=\bigO\left(\frac1{t^2}\right)
        \quad \text{and} \quad
        |f(x(t))-f(x^*)|=\bigO\left(\frac1{t^2}\right).
    \]
\end{itemize}
\end{theorem}
\begin{proof}
$(i)$ Denote $z(t)=(x(t),\lambda(t))$ and
$z^*=(x^*,\lambda^*)\in\Omega$. Let $\E_{z^*}(t)$ be the energy function
defined in \eqref{Edef}. By arguments similar to those in
\cite[Lemma 3.1]{BotJDE}, we obtain
\begin{eqnarray}\label{eq_dote}
&& \dot \E_{z^*}(t)
=2\theta^2t(\La_\beta(x(t),\lambda^*)- \La_\beta(x^*,\lambda(t)))
-\xi\theta t\norm{\dot z(t)}^2 \\
&&\qquad-\theta t\left(
\ip{\nabla f(x(t))+A^\top\lambda^*}{x(t)-x^*}
+\beta\langle A^\top(Ax(t)-b),x(t)-x^*\rangle
\right).\nonumber
\end{eqnarray}
Since $Ax^*=b$, we have
\begin{eqnarray*}
&& \ip{\nabla f(x(t))+A^\top\lambda^*}{x(t)-x^*}
+\beta\langle A^\top(Ax(t)-b),x(t)-x^*\rangle\\
&&\quad =  \ip{\nabla f(x(t))}{x(t)-x^*}
+ \ip{ \lambda^*}{Ax(t)-b}
+\beta\|Ax(t)-b\|^2\\
&&\quad =
D_f(x^*,x(t))
+\bigl(\La_\beta(x(t),\lambda^*)- \La_\beta(x^*,\lambda(t))\bigr)
+\frac{\beta}{2}\|Ax(t)-b\|^2,
\end{eqnarray*}
where $D_f(x^*,x)$ is the Bregman-distance function defined in
\eqref{eq_Df}. Combining this identity with \eqref{eq_dote}, we obtain
\begin{equation}\label{eq_Eine}
 \dot \E_{z^*}(t)
\leq
(2\theta-1)\theta t
\bigl(\La_\beta(x(t),\lambda^*)- \La_\beta(x^*,\lambda(t))\bigr)
-\xi\theta t\norm{\dot z(t)}^2
-\theta t D_f(x^*,x(t)).
\end{equation}
Here we have dropped the additional nonpositive term
$-\frac{\beta\theta}{2}t\|Ax(t)-b\|^2$.

Under Assumption \ref{ass1}, we have $2\theta-1\leq0$ and $\xi\geq0$.
Moreover, by the saddle point property,
\[
    \La_\beta(x(t),\lambda^*)- \La_\beta(x^*,\lambda(t))\ge0,
\]
and by the convexity of $f$, $D_f(x^*,x(t))\ge0$. Hence
$\dot \E_{z^*}(t)\leq0$. Since $\E_{z^*}(t)\geq0$, the function
$\E_{z^*}$ is nonincreasing and nonnegative. In particular,
\[
    \sup_{t\geq t_0}\E_{z^*}(t)<+\infty.
\]
Integrating \eqref{eq_Eine} over $[t_0,+\infty)$ yields the estimates in
$(i)$.

$(ii)$ Since $\E_{z^*}$ is nonincreasing, we have
$\E_{z^*}(t)\le \E_{z^*}(t_0)$ for all $t\ge t_0$. It follows from
\eqref{Edef} that
\begin{equation*}
 \norm{z(t)-z^*+\theta t\dot z(t)}^2\le 2\E_{z^*}(t_0).
\end{equation*}
Applying Lemma \ref{le_bounde} with $\bar z=z^*$ and $g(t)=\theta t$, we conclude
that $z(t)$ is bounded on $[t_0,+\infty)$ and that
\[
    \sup_{t\geq t_0}\left\|t\dot{z}(t)\right\|<+\infty.
\]
This proves $(ii)$.

$(iii)$ The proof of $(iii)$ follows from the same arguments as in
\cite[Theorem 3.4]{BotJDE}.
\end{proof}

\begin{remark}
Under the parameter condition $\alpha\ge 3$ and
$\theta\in[1/(\alpha-1),1/2]$, Theorem~\ref{th_basic} gives the boundedness
of  trajectory and the velocity estimate
$\|(\dot x(t),\dot\lambda(t))\|=\bigO(t^{-1})$. This slightly improves the
corresponding trajectory estimates in \cite{BotJDE}, which were obtained
under the noncritical condition $\alpha>3$ and
$\theta\in(1/(\alpha-1),1/2]$.
\end{remark}

\begin{remark}
A key point of Theorem~\ref{th_basic} is the additional weighted
integrability estimate
\[
    \int_{t_0}^{+\infty} tD_f(x^*,x(t))\,dt<+\infty.
\]
This estimate is not a direct consequence of the usual primal-dual gap
bound. It will be crucial in the subsequent trajectory convergence analysis,
since it allows us to replace the cocoercivity-based gradient estimate used
in the Lipschitz-gradient setting.
\end{remark}

We next prove the trajectory convergence of the dynamic \eqref{dy_main}. The main
difficulty is to identify the limiting stationarity relation without using
the Lipschitz continuity of $\nabla f$. Our proof proceeds in three steps.
First, the weighted integrability estimate
$\int_{t_0}^{+\infty}tD_f(x^*,x(t))\,dt<+\infty$ is converted into the
pointwise convergence of the Bregman distance. Second, this convergence is
used to prove the convergence of $\nabla f(x(t))$ along the primal
trajectory. Third, the first equation of \eqref{dy_main} is regarded as a
linear differential equation for the dual image
$A^\top(\lambda(t)-\lambda^*)$. These ingredients allow us to identify all
cluster points of the trajectory as primal-dual solutions and then prove the
convergence of the whole trajectory.

\begin{lemma}\label{lem:BregmanToZero}
Let $x(t)$ be the primal trajectory of the dynamic \eqref{dy_main}. Fix
$x^*\in\mathcal S$ and denote $\Phi(x)=D_f(x^*,x)$, where
$D_f(x^*,x)$ is defined in \eqref{eq_Df}. Then
\[
    \Phi(x(t))\to0, \quad\text{as } t\to+\infty.
\]
\end{lemma}

\begin{proof}
By Theorem \ref{th_basic}, $x(t)$ is bounded on $[t_0,+\infty)$. Hence
$
    K:=\overline{\{x(t):t\ge t_0\}}
$
is compact. Since $f$ is continuously differentiable, $\Phi$ is continuous.
Therefore, $\Phi$ is uniformly continuous on $K$. By Theorem \ref{th_basic}, we also have
\begin{equation}\label{BregIntAgain}
    \int_{t_0}^{+\infty}t\Phi(x(t))\,dt<+\infty,
\end{equation}
and there exists $C>0$ such that
\begin{equation}\label{xDotBound}
    \norm{\dot x(t)}\le \frac{C}{t},\qquad \forall t\ge t_0.
\end{equation}
We prove the result by contradiction. Suppose that
$\Phi(x(t))$ does not converge to $0$ as $t\to+\infty$. Since
$\Phi(x(t))\ge0$, there exist $\varepsilon_0>0$ and a sequence
$t_k\to+\infty$ such that
\begin{equation}\label{spike}
    \Phi(x(t_k))\ge \varepsilon_0,\qquad \forall k\ge1.
\end{equation}
By the uniform continuity of $\Phi$ on $K$, there exists $\delta_0>0$ such
that, for all $u,v\in K$,
\begin{equation}\label{UC}
    \norm{u-v}\le\delta_0
    \quad \Longrightarrow\quad
    |\Phi(u)-\Phi(v)|\le\frac{\varepsilon_0}{2}.
\end{equation}
Choose $\rho>0$ sufficiently small such that $C\ln(1+\rho)\le \delta_0$.
For every $s\in[t_k,(1+\rho)t_k]$, \eqref{xDotBound} gives
\[
\norm{x(s)-x(t_k)}
\le \int_{t_k}^{s}\norm{\dot x(\tau)}\,d\tau
\le C\int_{t_k}^{(1+\rho)t_k}\frac{d\tau}{\tau}
=C\ln(1+\rho)
\le\delta_0.
\]
Therefore, by \eqref{UC} and \eqref{spike},
\[
    \Phi(x(s))\ge\Phi(x(t_k))-|\Phi(x(t_k))-\Phi(x(s))|
    \ge \frac{\varepsilon_0}{2},
    \qquad \forall s\in[t_k,(1+\rho)t_k].
\]
Consequently,
\begin{equation*}
\begin{aligned}
\int_{t_0}^{+\infty}t\Phi(x(t))\,dt
\ge \int_{t_k}^{(1+\rho)t_k}t\Phi(x(t))\,dt
\ge \frac{\varepsilon_0}{2}\int_{t_k}^{(1+\rho)t_k}t\,dt
=\frac{\varepsilon_0}{4}\bigl((1+\rho)^2-1\bigr)t_k^2.
\end{aligned}
\end{equation*}
Letting $t_k\to+\infty$ contradicts \eqref{BregIntAgain}. Hence
$\Phi(x(t))\to0$ as $t\to+\infty$.
\end{proof}

Lemma~\ref{lem:BregmanToZero} shows that the Bregman distance vanishes along
the trajectory. The next step is to convert this scalar convergence into the
convergence of the gradient.

\begin{lemma}\label{lem:gradConv}
Let $x(t)$ be the primal trajectory of the dynamic \eqref{dy_main}, and let
$x^*\in\mathcal S$. Then
\[
    \nabla f(x(t))\to\nabla f(x^*),
    \qquad \text{as } t\to+\infty.
\]
\end{lemma}

\begin{proof}
Since $x(t)$ is bounded on $[t_0,+\infty)$, every sequence $t_k\to+\infty$ admits a subsequence,
still denoted by $t_k$, such that $x(t_k)$ converges to some $\bar x\in\R^n$.
By Lemma \ref{lem:BregmanToZero} and the continuity of $\Phi$, we have
$\Phi(\bar x)=0$. Therefore,
\begin{equation}\label{BregCluster2}
    f(x^*)-f(\bar x)-\ip{\nabla f(\bar x)}{x^*-\bar x}=0.
\end{equation}
By the convexity of $f$,
\begin{equation}\label{convsub}
    f(y)\ge f(\bar x)+\ip{\nabla f(\bar x)}{y-\bar x},
    \qquad \forall y\in\R^n.
\end{equation}
Combining \eqref{BregCluster2} with \eqref{convsub}, we obtain
\[
    f(y)\ge f(x^*)+\ip{\nabla f(\bar x)}{y-x^*},
    \qquad \forall y\in\R^n.
\]
It follows from the convexity of $f$ that $\nabla f(\bar x)\in\partial f(x^*)$. Since $f$ is
differentiable at $x^*$, we have 
\begin{equation}\label{gradCluster}
    \nabla f(\bar x)=\nabla f(x^*).
\end{equation}

We now prove that $\nabla f(x(t))\to\nabla f(x^*)$ as $t\to+\infty$. Suppose, to the contrary, that  there are $\varepsilon_0>0$ and $t_k\to+\infty$ such that
$
    \norm{\nabla f(x(t_k))-\nabla f(x^*)}\ge \varepsilon_0.
$
Since $x(t)$ is bounded on $[t_0,+\infty)$, after passing to a subsequence of $\{t_k\}_{k\geq 1}$, there exist $t_{k_j}\to+\infty$ and $\bar x\in\mathbb R^n$ such that
$x(t_{k_j})\to\bar x$. By \eqref{gradCluster},
$\nabla f(\bar x)=\nabla f(x^*)$. Hence, by the continuity of $\nabla f$, we have
\[
    \nabla f(x(t_{k_j}))\to \nabla f(\bar x)=\nabla f(x^*),
\]
which contradicts
$
    \|\nabla f(x(t_{k_j}))-\nabla f(x^*)\|\ge \varepsilon_0.
$
Thus $\nabla f(x(t))\to\nabla f(x^*)$ as $t\to+\infty$.
\end{proof}

Lemma~\ref{lem:gradConv} identifies the limiting gradient of the primal
trajectory. To complete the stationarity relation, it remains to control the
dual contribution $A^\top\lambda(t)$.

\begin{lemma}\label{lem:AtransposeLambda}
Let $(x(t),\lambda(t))$ be the trajectory of the dynamic \eqref{dy_main}, and $(x^*,\lambda^*)\in\Omega$. Then
\[
    A^\top\lambda(t)\to A^\top\lambda^*,
    \qquad \text{as } t\to+\infty.
\]
\end{lemma}

\begin{proof}
Set
\[
    y(t):=A^\top(\lambda(t)-\lambda^*),
    \quad
    a(t):=\nabla f(x(t))-\nabla f(x^*),
    \quad
    r(t):=\beta A^\top(Ax(t)-b).
\]
Using the first equation of \eqref{dy_main} and the relation
$\nabla f(x^*)+A^\top\lambda^*=0$, we obtain
\[
    \ddot x(t)+\frac{\alpha}{t}\dot x(t)+a(t)+y(t)+\theta t\dot y(t)+r(t)=0.
\]
Equivalently,
\[
    \theta t\dot y(t)+y(t)
    =-a(t)-r(t)-\ddot x(t)-\frac{\alpha}{t}\dot x(t).
\]
Let $p:=\frac{1}{\theta}$. Since $\theta\le \frac{1}{2}$, one has $p\ge2$.
Dividing the above equality by $\theta t$ gives
\[
    \dot y(t)+\frac{p}{t}y(t)
    =-\frac{p}{t}\left(a(t)+r(t)+\ddot x(t)+\frac{\alpha}{t}\dot x(t)\right).
\]
Multiplying both sides by the integrating factor $t^p$, we have
\[
    \frac{d}{dt}\left(t^py(t)\right)
    =
    -pt^{p-1}\left(a(t)+r(t)+\ddot x(t)+\frac{\alpha}{t}\dot x(t)\right).
\]
Integrating over $[t_0,t]$ and dividing by $t^p$, we get
\begin{equation}\label{eq_yt}
\begin{aligned}
y(t)=&\left(\frac{t_0}{t}\right)^py(t_0)
-p t^{-p}\int_{t_0}^{t}s^{p-1}a(s)\,ds
-p t^{-p}\int_{t_0}^{t}s^{p-1}r(s)\,ds  \\
&-p t^{-p}\int_{t_0}^{t}s^{p-1}\ddot x(s)\,ds
-p\alpha t^{-p}\int_{t_0}^{t}s^{p-2}\dot x(s)\,ds.
\end{aligned}
\end{equation}
We prove that each term on the right-hand side of \eqref{eq_yt} tends to
zero as $t\to+\infty$.

The first term clearly tends to zero. For the second term,  Lemma \ref{lem:gradConv} gives $\lim_{t+\infty}\|a(t)\|\to0$.
Then, it follows from L'H{\^o}pital's rule that
\[ \lim_{t+\infty} \left\|t^{-p}\int_{t_0}^{t}s^{p-1}a(s)\,ds\right\|\leq   \lim_{t+\infty} t^{-p}\int_{t_0}^{t}s^{p-1}\|a(s)\|\,ds=\lim_{t+\infty}\frac{\|a(t)\|}{p}=0.
\]
This implies that the second term  tends to zero. By the definition of $r(t)$ and Theorem \ref{th_basic}, we have $\|r(t)\|\leq \beta\|A\|\|Ax(t)-b\|\to 0$ as $t\to+\infty$, then the same argument   gives
that $p t^{-p}\int_{t_0}^{t}s^{p-1}r(s)\,ds$ tends to zero.
It remains to estimate the last two terms. Integrating by parts gives
\[
 t^{-p} \int_{t_0}^{t}s^{p-1}\ddot x(s)\,ds=\frac{1}{t}\dot x(t)-\frac{t_0^{p-1}}{t^p}\dot x(t_0)
-(p-1) t^{-p} \int_{t_0}^{t}s^{p-2}\dot x(s)\,ds.
\]
By Theorem \ref{th_basic}, $\|\dot x(t)\|=\bigO(1/t)$. Then together with $p\geq 2$ implies $\frac{1}{t}\dot x(t)\to0$ and $
    \frac{t_0^{p-1}}{t^p}\dot x(t_0)\to0$. Moreover, there exists $C>0$ such that $\|\dot{x}(t)\|\le \frac{C}{t}$, and then
\[
    \left\|t^{-p} \int_{t_0}^{t}s^{p-2}\dot x(s)\,ds\right\|
    \le
    C t^{-p} \int_{t_0}^{t}s^{p-3}\,ds\to0.
\]
Therefore, the last two terms in
\eqref{eq_yt} also tend to zero.

Consequently, the right-hand side of \eqref{eq_yt} tends to zero as
$t\to+\infty$, and hence $y(t)\to0$. By the definition of $y(t)$, this yields the result.
\end{proof}

Combining Lemma~\ref{lem:gradConv} and Lemma~\ref{lem:AtransposeLambda}, we
now state the main trajectory convergence result.

\begin{theorem}\label{thm:main}
Let $(x(t),\lambda(t))$ be the trajectory of the dynamic \eqref{dy_main}.
Then there exists $(x^*, \lambda^*)\in \Omega$ such that
\[
    (x(t),\lambda(t))\to (x^*, \lambda^*)
    \qquad \text{as } t\to+\infty.
\]
\end{theorem}

\begin{proof}
Set $z(t)=(x(t),\lambda(t))$. By Theorem~\ref{th_basic}, the trajectory
$z(t)$ is bounded on $[t_0,+\infty)$. Hence $z(t)$ has at least one cluster
point. We prove the convergence in two steps.

\emph{Step 1: Every cluster point belongs to $\Omega$.}

Let $(\bar x,\bar\lambda)$ be an arbitrary cluster point of
$(x(t),\lambda(t))$. Then there exists a sequence $t_k\to+\infty$ such that
\[
    (x(t_k),\lambda(t_k))\to(\bar x,\bar\lambda).
\]
By Theorem~\ref{th_basic}, we have
\[
    Ax(t)-b\to0,
    \qquad \text{as } t\to+\infty.
\]
Passing to the limit along $t_k$, we obtain $A\bar x=b$.

Next, fix an arbitrary point $(x^*,\lambda^*)\in\Omega$. Since
$x(t_k)\to\bar x$ and $\nabla f$ is continuous, it follows from Lemma
\ref{lem:gradConv} that
\[
   \nabla f(\bar x)
   =
   \lim_{k\to+\infty}\nabla f(x(t_k))
   =
   \nabla f(x^*).
\]
Moreover, by Lemma~\ref{lem:AtransposeLambda},
\[
A^\top\lambda(t)\to A^\top\lambda^*.
\]
Passing to the limit along $t_k$ and using $\lambda(t_k)\to\bar\lambda$, we
obtain
\[
    A^\top\bar\lambda=A^\top\lambda^*.
\]
Since $(x^*,\lambda^*)\in\Omega$, we conclude that
\[
    \nabla f(\bar x)+A^\top\bar\lambda
    =
    \nabla f(x^*)+A^\top\lambda^*
    =
    0.
\]
Together with $A\bar x=b$, this proves $(\bar x,\bar\lambda)\in\Omega$.
Therefore, every cluster point of $(x(t),\lambda(t))$ belongs to $\Omega$.

\emph{Step 2: The trajectory has a unique cluster point.}

Assume, by contradiction, that $z(t)$ has two distinct cluster points
$z_1=(x_1,\lambda_1)$ and $z_2=(x_2,\lambda_2)$. Then there exist two
sequences $s_k\to+\infty$ and $t_k\to+\infty$ such that
\begin{equation}\label{eq_con_z12}
    z(s_k)\to z_1,
    \qquad
    z(t_k)\to z_2.
\end{equation}
By Step 1, both $z_1$ and $z_2$ belong to $\Omega$.

From the proof of Theorem \ref{th_basic} $(i)$, for every $z^*\in\Omega$,
the limit
\[
    \lim_{t\to+\infty}\E_{z^*}(t)
\]
exists, where $\E_{z^*}(t)$ is defined in \eqref{Edef}. Hence both
$\lim_{t\to+\infty}\E_{z_1}(t)$ and
$\lim_{t\to+\infty}\E_{z_2}(t)$ exist. Therefore,
\[
    H(t):=\frac{1}{\theta}\bigl(\E_{z_1}(t)-\E_{z_2}(t)\bigr)
\]
admits a finite limit as $t\to+\infty$. By the definition of
$\E_{z^*}(t)$, all terms independent of $z^*$ cancel in the difference, and
we obtain
\begin{eqnarray}\label{eq_con_Ht}
	  H(t)
    &=&\theta t^2(\ip{\lambda_1-\lambda_2}{Ax(t)-b})+
    \frac{\alpha-1}{2}
    \left(
    \|z(t)-z_1\|^2-\|z(t)-z_2\|^2
    \right)\nonumber \\
    &&- t\langle \dot z(t),z_1-z_2\rangle.
\end{eqnarray}
Since $z_1,z_2\in\Omega$, we have $Ax_1=b$. Moreover, by Proposition
\ref{prop_App1}, $A^{\top}\lambda_1=A^{\top}\lambda_2$. Hence
\[
\ip{\lambda_1-\lambda_2}{Ax(t)-b}
=
\ip{A^{\top}(\lambda_1-\lambda_2)}{x(t)-x_1}
=0.
\]
Introduce the auxiliary function
\[
h(t):=\frac12\bigl(\|z(t)-z_1\|^2-\|z(t)-z_2\|^2\bigr).
\]
Then $
    \dot{h}(t)=-\langle \dot{z}(t),z_1-z_2\rangle.
$
Therefore, from \eqref{eq_con_Ht}, we have
\[
    (\alpha-1)h(t)+t\dot{h}(t)=H(t).
\]
Multiplying both sides by $t^{\alpha-2}$ gives
\[
    \frac{d}{dt}\left(t^{\alpha-1}h(t)\right)
    =
    t^{\alpha-2}H(t).
\]
Integrating over $[t_0,t]$, we get
\[
    t^{\alpha-1}h(t)-t_0^{\alpha-1}h(t_0)
    =
    \int_{t_0}^{t}s^{\alpha-2}H(s)\,ds.
\]
Thus
\[
h(t)
=
\frac{t_0^{\alpha-1}}{t^{\alpha-1}}h(t_0)
+
\frac{1}{t^{\alpha-1}}\int_{t_0}^{t}s^{\alpha-2}H(s)\,ds.
\]
Since $H(t)$ admits a finite limit as $t\to+\infty$ and $\alpha\ge3$, it
follows from l'H\^opital's rule that $h(t)$ has a finite limit. More
precisely,
\[
\lim_{t\to+\infty} h(t)
=
\lim_{t\to+\infty}
\frac{1}{t^{\alpha-1}}\int_{t_0}^{t}s^{\alpha-2}H(s)\,ds
=
\frac{1}{\alpha-1}\lim_{t\to+\infty}H(t),\quad \text{exists}.
\]

On the other hand, by the definition of $h(t)$ and \eqref{eq_con_z12}, we
have
\[
\lim_{s_k\to+\infty}h(s_k)
=
-\frac12\|z_1-z_2\|^2,
\qquad
\lim_{t_k\to+\infty}h(t_k)
=
\frac12\|z_1-z_2\|^2.
\]
Since $h(t)$ has a unique limit as $t\to+\infty$, the two subsequential
limits must be equal. Therefore,
\[
    -\frac12\|z_1-z_2\|^2
    =
    \frac12\|z_1-z_2\|^2,
\]
which implies $z_1=z_2$. This contradicts the assumption that $z_1$ and
$z_2$ are distinct cluster points. Hence the bounded trajectory
$z(t)=(x(t),\lambda(t))$ has a unique cluster point.

Combining Step 1 and Step 2, the bounded trajectory $z(t)$ has a unique
cluster point, and this cluster point belongs to $\Omega$. Therefore, there
exists $z^*=(x^*,\lambda^*)\in\Omega$ such that
\[
    (x(t),\lambda(t))\to(x^*, \lambda^*),
    \qquad t\to+\infty.
\]
The proof is complete.
\end{proof}

\begin{remark}
Theorem~\ref{thm:main} extends the existing trajectory convergence results
in two directions. The convergence result in \cite{BotJDE} was obtained
under the Lipschitz continuity of $\nabla f$ and the noncritical condition
$\alpha>3$. In contrast, our result includes the critical damping case
$\alpha=3$ and does not require the Lipschitz continuity of $\nabla f$. The
proof strategy is also different. Instead of relying on cocoercivity and
Opial's lemma as in \cite{BotJDE}, we use a Bregman-distance argument and
prove convergence through the existence and uniqueness of cluster points.
When the constraint is absent, the result reduces to the corresponding
trajectory convergence results for $\text{\rm (AVD)}_{\alpha}$ in
\cite{AttouchMP,Jang25}.
\end{remark}

\section{Improved convergence rates in the case $\alpha>3$}\label{sec3}

In this section, we investigate improved convergence rates for the objective
residual, the feasibility violation, and the velocity of the trajectory
$(x(t),\lambda(t))$ generated by the dynamic \eqref{dy_main}. The parameters
are assumed to satisfy the following stronger condition.

 \begin{center}
 \vspace{0.5em}
\setlength{\fboxsep}{6pt}
\setlength{\fboxrule}{0.4pt}
\fbox{%
\begin{minipage}{0.96\textwidth}
\begin{assumption}\label{ass2}
	The parameter $\alpha$, $\beta$, $\theta$ in the dynamic \eqref{dy_main} satisfy
	\[\alpha> 3, \quad \theta\in\left(\frac{1}{\alpha-1},\frac12\right),\quad  \beta\ge 0.\]
\end{assumption}
\end{minipage}}
\vspace{0.5em}
\end{center}

Compared with Assumption \ref{ass1}, Assumption \ref{ass2} excludes the
critical damping case and the two endpoint values of $\theta$. This strict
parameter condition allows us to use the weighted integral estimates in
Theorem~\ref{th_basic} together with Lemma~\ref{lem:tech} in Appendix A to
derive the improved rates $o(t^{-2})$ for the objective residual and the
feasibility violation, and $o(t^{-1})$ for the velocity.

\begin{theorem}\label{th_ocon}
Let $(x(t),\lambda(t))$ be a trajectory of the dynamic \eqref{dy_main} with
parameters satisfying Assumption \ref{ass2}. Then the following convergence
rates hold:
\begin{itemize}
    \item[$(i)$] Velocity:
    \[
        \norm{(\dot x(t), \dot \lambda(t))}
        =o\left(\frac1t\right).
    \]
    \item[$(ii)$] Objective residual and feasibility violation:
    \[
        |f(x(t))-f(x^*)|
        =o\left(\frac1{t^2}\right),
        \qquad
        \|Ax(t)-b\|
        =o\left(\frac1{t^2}\right).
    \]
\end{itemize}
\end{theorem}

\begin{proof}
Since the parameters satisfy Assumption \ref{ass2}, they also satisfy
Assumption \ref{ass1}. Therefore, all convergence results in
Section~\ref{sec2} are available. Let $z(t)=(x(t),\lambda(t))$. By
Theorem~\ref{thm:main}, there exists $z^*=(x^*,\lambda^*)\in\Omega$ such
that
\[
   \lim_{t\to+\infty} z(t) = z^*.
\]
Moreover, by Theorem~\ref{th_basic} $(ii)$, $t\|\dot z(t)\|$ is bounded on
$[t_0,+\infty)$. Hence
\[
\lim_{t\to+\infty}
\left|\ip{z(t)-z^*}{t\dot z(t)}\right|
\le
\sup_{t\ge t_0} t\|\dot z(t)\|
\cdot
\lim_{t\to+\infty}\|z(t)-z^*\|
=0.
\]
From the proof of Theorem~\ref{th_basic}, the limit of the energy function
exists, namely,
\[
    \lim_{t\to+\infty}\E_{z^*}(t)
    \qquad \text{exists}.
\]
Using the definition of $\E_{z^*}(t)$ in \eqref{Edef}, the convergence
$z(t)\to z^*$, and the preceding estimate, we obtain that
\[
\lim_{t\to+\infty}
\theta^2 t^2\left(
\La_\beta(x(t),\lambda^*)-\La_\beta(x^*,\lambda(t))
+\frac{1}{2}\norm{\dot z(t)}^2
\right)
\qquad \text{exists}.
\]
Since $\theta\in\left(\frac{1}{\alpha-1},\frac{1}{2}\right)$, Theorem
\ref{th_basic} $(i)$ gives
\[
\int_{t_0}^{+\infty}
t(\La_\beta(x(t),\lambda^*)-\La_\beta(x^*,\lambda(t)))\,dt<+\infty
\]
and
\[
\int_{t_0}^{+\infty} t\norm{\dot z(t)}^2\,dt<+\infty.
\]

Set
\[
G(t):=
\La_\beta(x(t),\lambda^*)-\La_\beta(x^*,\lambda(t))
+\frac12\|\dot z(t)\|^2.
\]
Then $G(t)\ge 0$ and
\[
    \lim_{t\to+\infty} t^2G(t)\ \text{exists},
    \qquad
    \int_{t_0}^{+\infty}  tG(t)\,dt <+\infty.
\]
It follows from Lemma~\ref{lem_app_o} that
$
    G(t)=o\left({t^{-2}}\right).
$
By the definition of $G(t)$, we obtain
\[
    \|\dot z(t)\|=o\left(\frac1t\right),
\]
which proves $(i)$. Moreover,
\begin{equation}\label{eq_laro}
\La_\beta(x(t),\lambda^*)-\La_\beta(x^*,\lambda(t))
=o\left(\frac{1}{t^{2}}\right).
\end{equation}

Next, we prove $(ii)$. Let
\[
    r(t):=Ax(t)-b.
\]
Then
\[
    r(t)+\theta t\dot r(t)
    =
    A(x(t)+\theta t\dot x(t))-b.
\]
By the second equation of the dynamic \eqref{dy_main}, we have
\[
    r(t)+\theta t\dot r(t)
    =
    \ddot\lambda(t)+\frac{\alpha}{t}\dot\lambda(t).
\]
Therefore,
\[
\begin{aligned}
\frac{1}{t^{\alpha-1}}
\int_{t_0}^{t}s^\alpha\bigl(r(s)+\theta s\dot r(s)\bigr)\,ds
&=
\frac{1}{t^{\alpha-1}}
\int_{t_0}^{t}s^\alpha
\left(\ddot\lambda(s)+\frac{\alpha}{s}\dot\lambda(s)\right)\,ds  \\
&=
\frac{1}{t^{\alpha-1}}
\int_{t_0}^{t}\frac{d}{ds}\bigl(s^\alpha\dot\lambda(s)\bigr)\,ds  \\
&=
t\dot\lambda(t)
-
\frac{t_0^\alpha}{t^{\alpha-1}}\dot\lambda(t_0).
\end{aligned}
\]
By $(i)$, $t\dot\lambda(t)\to0$ as $t\to+\infty$. Since $\alpha>3$, we also
have
\[
    \frac{t_0^\alpha}{t^{\alpha-1}}\dot\lambda(t_0)\to0.
\]
Thus
\[
\lim_{t\to+\infty}
\frac{1}{t^{\alpha-1}}
\int_{t_0}^{t}s^\alpha\bigl(r(s)+\theta s\dot r(s)\bigr)\,ds
=0.
\]
It follows from Lemma~\ref{lem:tech} with $\delta(t)=1$ that
\[
    \|Ax(t)-b\|=\|r(t)\|=o(t^{-2}).
\]

It remains to estimate the objective residual. By the definition of the
augmented Lagrangian residual and the feasibility of $x^*$, we have
\[
\La_\beta(x(t),\lambda^*)-\La_\beta(x^*,\lambda(t))
=
f(x(t))-f(x^*)+\langle \lambda^*,Ax(t)-b\rangle
+\frac{\beta}{2}\|Ax(t)-b\|^2.
\]
Hence, from \eqref{eq_laro}, we have
\[
\begin{aligned}
|f(x(t))-f(x^*)|
&\le
\La_\beta(x(t),\lambda^*)-\La_\beta(x^*,\lambda(t))
+\|\lambda^*\|\|Ax(t)-b\|+\frac{\beta}{2}\|Ax(t)-b\|^2\\
&=o\left(\frac1{t^2}\right).
\end{aligned}
\]
This proves $(ii)$, and the proof is complete.
\end{proof}

\begin{remark}
Theorem~\ref{th_ocon} improves the existing $\bigO(t^{-2})$ convergence
rates for the Nesterov accelerated primal-dual dynamic \eqref{dy_main} studied in
\cite{ZengTAC,BotJDE,HeSICON,HeAA}. More precisely, the objective residual
and the feasibility violation are improved from $\bigO(t^{-2})$ to
$o(t^{-2})$, and the velocity estimate is improved from $\bigO(t^{-1})$ to
$o(t^{-1})$. In the unconstrained case $A=0$ and $b=0$, the result reduces to the
corresponding $o(t^{-2})$ rate for the classical
$\text{\rm (AVD)}_{\alpha}$ dynamic \cite{MayTJM,AttouchJEMS}.
\end{remark}

  \section{Extensions to time-scaled Nesterov accelerated primal-dual dynamical system}\label{sec4}

In this section, we show that the convergence results developed above can be
extended to the following time-scaled Nesterov accelerated primal-dual dynamical system:
\begin{equation}\label{dy_scaled}
\begin{cases}
\ddot{x}(t)+\dfrac{\alpha}{t}\dot{x}(t)
+\delta(t)\left(\nabla f(x(t))
+A^\top\bigl(\lambda(t)+\theta t\dot{\lambda}(t)\bigr)
+\beta A^\top(Ax(t)-b)\right)=0,\\[0.7em]
\ddot{\lambda}(t)+\dfrac{\alpha}{t}\dot{\lambda}(t)
-\delta(t)\bigl(A(x(t)+\theta t\dot{x}(t))-b\bigr)=0,
\end{cases}
\end{equation}
where $\delta:[t_0,+\infty)\to(0,+\infty)$ is a continuously differentiable
time-scaling function. The dynamic \eqref{dy_scaled} can be viewed as a
time-rescaled version of the primal-dual dynamic \eqref{dy_main}. Related
time-scaled primal-dual dynamical systems were studied in
\cite{HulettAMO,HeAA,AttouchJOTA}, where convergence rates depending on
$\delta(t)$ and weak convergence of trajectories were obtained under the
Lipschitz continuity of $\nabla f$.

We work in the finite-dimensional setting and assume only that $f$ is convex
and continuously differentiable. Throughout this subsection, we impose the
following parameter and scaling assumptions.

 \begin{center}
 \vspace{0.5em}
\setlength{\fboxsep}{6pt}
\setlength{\fboxrule}{0.4pt}
\fbox{%
\begin{minipage}{0.96\textwidth}
\begin{assumption}\label{ass3}
	The parameter $\alpha$, $\beta$, $\theta$ in the dynamic \eqref{dy_scaled} satisfy
	 \[\alpha\geq 3,\qquad
   \theta\in\left[\frac{1}{\alpha-1},\frac12\right],  \qquad  \beta\ge0.\]
   Moreover, $\delta:[t_0,+\infty)\to[\delta_0,+\infty)$  is a continuously differentiable function with $\delta_0>0$ and satisfy
\[
   \inf_{t\ge t_0}\frac{t\dot\delta(t)}{\delta(t)} >-\infty,\quad  \sup_{t\ge t_0}\frac{t\dot\delta(t)}{\delta(t)}
    \leq
  \frac{1}{\theta}-2, \qquad \lim_{t\to+\infty}t^2\delta(t) = +\infty.
\]
   \end{assumption}
\end{minipage}
}
\vspace{0.5em}
\end{center}

For fixed $z^*=(x^*,\lambda^*)\in\Omega$, let
$z(t)=(x(t),\lambda(t))$ be a trajectory of the dynamic \eqref{dy_scaled}. We recall the
corresponding Lyapunov energy function used in \cite{HulettAMO}:
\[
    \E^{\delta}_{z^*}(t)
    =
    \theta^2t^2\delta(t)(\La_\beta(x(t),\lambda^*)- \La_\beta(x^*,\lambda(t)))
    +\frac{1}{2}\norm{v(t)}^2
    +\frac{\xi}{2}\norm{z(t)-z^*}^2,
\]
where
\[
    v(t)=z(t)-z^*+\theta t\dot z(t),
    \qquad
    \xi=\theta\alpha-\theta-1.
\]
By the proof of \cite[Lemma 3.1]{HulettAMO}, one has
\begin{eqnarray*}
 \dot \E^{\delta}_{z^*}(t)
&=&
\theta^2t(2\delta(t)+t\dot{\delta}(t))
(\La_\beta(x(t),\lambda^*)- \La_\beta(x^*,\lambda(t)))
-\xi\theta t\norm{\dot z(t)}^2\\
&&-\theta t\delta(t)\left(
\ip{\nabla f(x(t))+A^\top\lambda^*}{x(t)-x^*}
+\beta\langle A^\top(Ax(t)-b),x(t)-x^*\rangle
\right).
\end{eqnarray*}
Using the same decomposition as in the proof of Theorem~\ref{th_basic}, we
obtain
\begin{eqnarray*}
\dot \E^{\delta}_{z^*}(t)
&\leq&
\theta^2t\left(\left(2-\frac{1}{\theta}\right)\delta(t)+t\dot{\delta}(t)\right)
(\La_\beta(x(t),\lambda^*)- \La_\beta(x^*,\lambda(t)))\\
&&-\xi\theta t\norm{\dot z(t)}^2
-\theta t\delta(t)D_f(x^*,x(t)).
\end{eqnarray*}
By Assumption~\ref{ass3}, the right-hand side is negative. Hence
$\E^{\delta}_{z^*}$ is nonincreasing. Arguing as in Theorem~\ref{th_basic} and \cite[Theorem 3.2-Theorem 3.3]{HulettAMO}, we obtain the following basic estimates.

\begin{theorem}\label{th_scal}
Let $(x(t),\lambda(t))$ be a trajectory of the dynamic \eqref{dy_scaled},
and   $(x^*,\lambda^*)\in\Omega$. Suppose that Assumption~\ref{ass3} holds.
Then the trajectory $(x(t),\lambda(t))$ is bounded on $[t_0,+\infty)$, and
the following estimates hold:
\[
    \|Ax(t)-b\|
    =
    \bigO\left(\frac1{t^2\delta(t)}\right),
    \qquad
    |f(x(t))-f(x^*)|
    =
    \bigO\left(\frac1{t^2\delta(t)}\right),
\]
\[
    \norm{(\dot x(t), \dot \lambda(t))}
    =
    \bigO\left(\frac1t\right),
    \qquad
    \int_{t_0}^{+\infty} t\delta(t)D_f(x^*,x(t))\,dt<+\infty.
\]
Moreover, if $\alpha>3$,
$\theta\in\left(\frac{1}{\alpha-1},\frac12\right)$ and $\sup_{t\ge t_0}\frac{t\dot\delta(t)}{\delta(t)}<
  \frac{1}{\theta}-2$, then
\[
    \int_{t_0}^{+\infty}
    t\norm{(\dot x(t),\dot \lambda(t))}^2\,dt<+\infty
\]
and
\[
\int_{t_0}^{+\infty}
t{\delta}(t) 
(\La_\beta(x(t),\lambda^*)- \La_\beta(x^*,\lambda(t)))\,dt
<+\infty.
\]
\end{theorem}

Since $\delta(t)\ge \delta_0>0$, Theorem~\ref{th_scal} gives
\[
    \int_{t_0}^{+\infty}tD_f(x^*,x(t))\,dt<+\infty.
\]
Together with $\|\dot x(t)\|=\bigO(t^{-1})$ and the boundedness of the
trajectory, the same Bregman-distance argument as in Section~\ref{sec2}
yields
\[
    D_f(x^*,x(t))\to0,
    \qquad \text{as } t\to+\infty.
\]
Hence, by the finite-dimensional compactness argument and the constancy of
$\nabla f$ on the primal solution set, we obtain
$
    \nabla f(x(t))\to \nabla f(x^*).
$

We next prove the convergence of the dual image. From the first equation of
\eqref{dy_scaled} and the relation
$\nabla f(x^*)+A^\top\lambda^*=0$, one obtains
\begin{eqnarray*}
&&\theta t\frac{d}{dt}\left(A^\top(\lambda(t)-\lambda^*)\right)
+
A^\top(\lambda(t)-\lambda^*)
\\
&&\quad=
-\bigl(\nabla f(x(t))-\nabla f(x^*)\bigr)
-\beta A^\top(Ax(t)-b)
-\frac{1}{\delta(t)}
\left(\ddot x(t)+\frac{\alpha}{t}\dot x(t)\right).
\end{eqnarray*}
The first two terms on the right-hand side tend to zero by the preceding
estimates. Therefore, as in the proof of Lemma~\ref{lem:AtransposeLambda},
it remains to verify that
\begin{equation}\label{eq_new1}
\frac{1}{t^p}\int_{t_0}^{t}
\frac{s^{p-1}}{\delta(s)}
\left(\ddot x(s)+\frac{\alpha}{s}\dot x(s)\right)\,ds
\to 0,
\qquad \text{as } t\to+\infty,
\end{equation}
where $p=\frac{1}{\theta}$. Once \eqref{eq_new1} is proved, the same
integrating  argument as in Lemma~\ref{lem:AtransposeLambda} gives
\[
  A^\top\lambda(t)\to A^\top\lambda^*,
    \qquad \text{as } t\to+\infty.
\]

We now prove \eqref{eq_new1}. By integration by parts,
\[
\begin{aligned}
&\int_{t_0}^{t}
\frac{s^{p-1}}{\delta(s)}
\left(\ddot x(s)+\frac{\alpha}{s}\dot x(s)\right)\,ds  \\
&\quad =
\frac{t^{p-1}}{\delta(t)}\dot x(t)
-\frac{t_0^{p-1}}{\delta(t_0)}\dot x(t_0)
+\int_{t_0}^{t}
\frac{s^{p-2}}{\delta(s)}
\left(\alpha-p+1+\frac{s\dot\delta(s)}{\delta(s)}\right)\dot x(s)\,ds .
\end{aligned}
\]
By Assumption~\ref{ass3}, we have $\delta(s)\ge\delta_0>0$ and
$
    \left|\alpha-p+1+\frac{s\dot\delta(s)}{\delta(s)}\right|\le C
$
for some constant $C>0$. Together with $\|\dot x(t)\|=\bigO(t^{-1})$ and
$p\ge2$, this yields
\[
\begin{aligned}
&\frac{1}{t^p}
\left\|
\int_{t_0}^{t}
\frac{s^{p-1}}{\delta(s)}
\left(\ddot x(s)+\frac{\alpha}{s}\dot x(s)\right)\,ds
\right\| \\
&\quad\le
\frac{\|\dot{x}(t)\|}{\delta_0t}
+\frac{t_0^{p-1}\|\dot x(t_0)\|}{\delta(t_0)t^p}
+\frac{C}{\delta_0}\frac{1}{t^p}\int_{t_0}^{t}s^{p-2}\|\dot{x}(s)\|\,ds
\to 0.
\end{aligned}
\]
Thus \eqref{eq_new1} holds, and consequently
$
A^\top(\lambda(t)-\lambda^*)\to0.
$
Since Theorem~\ref{th_scal} also gives
$
    \|Ax(t)-b\|=\bigO(1/(t^2\delta(t))),
$
and Assumption~\ref{ass3} implies $t^2\delta(t)\to+\infty$, we have
$
    Ax(t)-b\to0.
$
Therefore, every cluster point of $(x(t),\lambda(t))$ satisfies both the
feasibility condition and the stationarity condition, and hence belongs to
$\Omega$.

Based on the above discussions, we now prove the following result.

\begin{theorem}\label{thm:scaled-convergence}
Let $(x(t),\lambda(t))$ be a global trajectory of the dynamic \eqref{dy_scaled}. Suppose
that Assumption~\ref{ass3} holds. Then there exists $(x^*,\lambda^*)\in\Omega$ such that
$
    (x(t),\lambda(t))\to(x^*,\lambda^*)
$
as $t\to+\infty$.
Moreover, if $\alpha>3$,
$\theta\in\left(\frac{1}{\alpha-1},\frac12 \right)$, and
$\sup_{t\ge t_0}\frac{t\dot\delta(t)}{\delta(t)}<\frac{1}{\theta}-2$, then
\[\norm{(\dot x(t), \dot \lambda(t))}=o\left(\frac1t\right)\] and
\[
    \|Ax(t)-b\|
    =
    o\left(\frac{1}{t^2\delta(t)}\right),\quad
    |f(x(t))-f(x^*)|
    =
    o\left(\frac{1}{t^2\delta(t)}\right).
\]
\end{theorem}

\begin{proof}
From the above discussions, every cluster point of $z(t)$ belongs to
$\Omega$. The remaining part is the same as in the proof of
Theorem~\ref{thm:main}: if
\[
h(t)=\frac12\|z(t)-z_1\|^2-\frac12\|z(t)-z_2\|^2,
\]
then the existence of the energy-difference limit implies the existence of
$\lim_{t\to+\infty}h(t)$. Passing to the limits along two subsequences
converging to $z_1$ and $z_2$, respectively, yields $z_1=z_2$. Hence the
trajectory has a unique cluster point. Since it is bounded in a
finite-dimensional space, it converges to some
$(x^*,\lambda^*)\in\Omega$.

Moreover, if $\alpha>3$,
$\theta\in\left(\frac{1}{\alpha-1},\frac12 \right)$, and
$\sup_{t\ge t_0}\frac{t\dot\delta(t)}{\delta(t)}<\frac{1}{\theta}-2$, then
Theorem~\ref{th_scal} gives
\[
\int_{t_0}^{+\infty} t\norm{(\dot x(t),\dot \lambda(t))}^2\,dt<+\infty,
\quad
\int_{t_0}^{+\infty} t\delta(t)
(\La_\beta(x(t),\lambda^*)-\La_\beta(x^*,\lambda(t)))\,dt<+\infty.
\]
By arguments similar to those in Theorem~\ref{th_ocon}, we have
\[
\norm{(\dot x(t), \dot \lambda(t))}=o\left(\frac1t\right),
\quad
\La_\beta(x(t),\lambda^*)-\La_\beta(x^*,\lambda(t))
=o\left(\frac{1}{t^2\delta(t)}\right).
\]
Multiplying both sides of the second equation of \eqref{dy_scaled} by
$t^\alpha$ and integrating over $[t_0,t]$ yields
\[
    \frac{1}{t^{\alpha-1}}
    \int_{t_0}^{t}s^\alpha\delta(s)
    \bigl(Ax(s)-b+\theta sA\dot x(s)\bigr)\,ds
    =
    t\dot\lambda(t)
    -
    \frac{t_0^\alpha}{t^{\alpha-1}}\dot\lambda(t_0).
\]
Since $\|\dot\lambda(t)\|=o(t^{-1})$, the right-hand side tends to zero.
Applying Lemma~\ref{lem:tech} with $r(t)=Ax(t)-b$ gives
\[
    \|Ax(t)-b\|= o\left(\frac{1}{t^2\delta(t)}\right).
\]
Then
\[
    |f(x(t))-f(x^*)|
    =
    o\left(\frac{1}{t^2\delta(t)}\right).
\]
The proof is complete.
\end{proof}

\begin{remark}
Compared with \cite{HulettAMO}, Theorem~\ref{thm:scaled-convergence} shows
that, in finite-dimensional spaces, the trajectory convergence of the
time-rescaled Nesterov accelerated  primal-dual dynamic can be obtained without assuming
Lipschitz continuity of $\nabla f$. In addition, the theorem improves the
existing $\bigO(1/(t^2\delta(t)))$ estimates in
\cite{HulettAMO,AttouchJOTA,HeAA} to
$o(1/(t^2\delta(t)))$ rates. When $A=0$ and $b=0$, the results reduce to the
corresponding convergence properties of the time-scaled
$\text{\rm (AVD)}_{\alpha}$ dynamic \cite{AttouchSIOPT}.
\end{remark}

\section{Conclusions and Perspectives}\label{sec5}

In this paper, we studied a Nesterov accelerated inertial primal-dual
dynamical system \eqref{dy_main} for the linearly constrained convex optimization problem \eqref{main_ques}. In a
finite-dimensional space setting, we proved the convergence of the whole trajectory
to a primal-dual solution under $\alpha\ge3$, without assuming the Lipschitz
continuity of $\nabla f$. This extends the trajectory convergence result of
Bo\c{t} and Nguyen \cite{BotJDE}, where the analysis requires Lipschitz
continuity of $\nabla f$ and the noncritical condition $\alpha>3$. The main tool is a Bregman-distance estimate, which replaces the usual
cocoercivity-based gradient residual estimate in the Lipschitz-gradient
setting.  We also derived improved convergence rates in the noncritical regime
$\alpha>3$ and $\theta\in(1/(\alpha-1),1/2)$. More precisely, the objective
residual and feasibility violation decay as $o(t^{-2})$, while the velocity
satisfies $o(t^{-1})$. These results improve the known $\bigO(t^{-2})$
estimates for related Nesterov-type primal-dual dynamics
\cite{ZengTAC,HeSICON,AttouchJOTA,BotJDE}, and reduce in the unconstrained
case to the $o(t^{-2})$ convergence behavior of
$\text{\rm (AVD)}_{\alpha}$ \cite{MayTJM}. We also indicated that the same
strategy applies to  time-scaled Nesterov accelerated primal-dual dynamic \eqref{dy_scaled} in
\cite{HulettAMO,HeAA} 

Future work includes the design of structure-preserving discretization that
retain the main continuous-time properties at the algorithmic level, including
 convergence of iterate and improved $o(k^{-2})$ rates. It would also
be interesting to investigate whether the approach can be extended to
composite objectives, nonsmooth terms through proximal dynamics, saddle point
problems, and primal-dual dynamical systems with more general damping, scaling, or
extrapolation coefficients. Another related direction is to understand which
parts of the present finite-dimensional argument can be carried over to
Hilbert spaces under suitable 
assumptions.

\appendix

\section{Auxiliary results}\label{AppA}
\setcounter{equation}{0}
\renewcommand{\theequation}{A.\arabic{equation}}
\renewcommand{\thesection}{\Alph{section}}

We collect here some auxiliary results used in the proof of the convergence
properties of the dynamical system \eqref{dy_main}.

\begin{proposition}\label{prop_App1}
Consider the linearly constrained problem \eqref{main_ques}, where
$f:\R^n\to\R$ is convex and differentiable. Then, for any
$(x_1^*,\lambda_1^*), (x_2^*,\lambda_2^*)\in \Omega$, one has
\[
\nabla f(x_1^*)=\nabla f(x_2^*)\quad \text{and} \quad
A^\top\lambda_1^*=A^\top\lambda_2^*.
\]
\end{proposition}

\begin{proof}
Since $(x_i^*,\lambda_i^*)\in \Omega$, $i=1,2$, the points $x_1^*$ and
$x_2^*$ are both optimal solutions. Denote the optimal value by $f^*$. Then
\[
f(x_1^*)=f(x_2^*)=f^*,\qquad Ax_1^*=Ax_2^*=b.
\]
For any $s\in[0,1]$, set $x_s=(1-s)x_1^*+sx_2^*$. Since the constraint is
affine, $Ax_s=b$. By the convexity of $f$,
\[
f(x_s)\le (1-s)f(x_1^*)+sf(x_2^*)=f^*.
\]
Together with $Ax_s=b$, this implies
\[
f(x_s)=f^*,\qquad \forall s\in[0,1].
\]
Therefore, the function
\[
\phi(s)=f(x_1^*+s(x_2^*-x_1^*))
\]
is constant on $[0,1]$. Since $f$ is differentiable,
\begin{equation}\label{eq_app1}
0=\phi'_+(0)
=
\left\langle \nabla f(x_1^*),x_2^*-x_1^*\right\rangle .
\end{equation}
By convexity, for every $y\in\R^n$,
$
f(y)\ge f(x_1^*)+
\left\langle \nabla f(x_1^*),y-x_1^*\right\rangle .
$
Using \eqref{eq_app1} and $f(x_1^*)=f(x_2^*)$, we get
\[
f(y)\ge f(x_2^*)+
\left\langle \nabla f(x_1^*),y-x_2^*\right\rangle ,
\qquad \forall y\in\R^n.
\]
Thus 
$
\nabla f(x_1^*)=\nabla f(x_2^*).
$

The assumption $(x_i^*,\lambda_i^*)\in \Omega$ gives
\[
A^\top\lambda_i^*=-\nabla f(x_i^*),\qquad i=1,2.
\]
Therefore,
$
A^\top\lambda_1^*=A^\top\lambda_2^*.
$
\end{proof}

\begin{remark}
The argument of Proposition~\ref{prop_App1} extends directly to real Hilbert
spaces when $f:\mathcal{H}\to\mathbb R$ is convex and Fr\'echet
differentiable. The proof only uses the affine feasibility of the segment
between two primal solutions and the gradient inequality in Hilbert spaces.
Hence the conclusion of \cite[Proposition A.4]{BotJDE} does not require the
Lipschitz continuity of $\nabla f$.
\end{remark}

\begin{lemma}\cite[Lemma 1]{HeNN}\label{le_bounde}
Assume that $z:[t_0,+\infty)\to \mathbb{R}^N$ is continuously
differentiable, $g:[t_0,+\infty)\to(0,+\infty)$ is continuously
differentiable, $\bar z\in \mathbb{R}^N$, $t_0>0$, and $C>0$. If
\[
    \left\|z(t)-\bar z+{g(t)}\dot{z}(t)\right\|^2\leq C,
\]
then $z(t)$ is bounded on $[t_0,+\infty)$ and
$\sup_{t\geq t_0}\left\|{g(t)}\dot{z}(t)\right\|<+\infty$.
\end{lemma}

\begin{lemma}\label{lem_app_o}
Let $g:[t_0,+\infty)\to[0,+\infty)$ be a continuous function such that
$\lim_{t\to+\infty}t^2g(t)$ exists and
$\int_{t_0}^{+\infty}tg(t)\,dt<+\infty$. Then
\[
    g(t)=o(t^{-2}).
\]
\end{lemma}

\begin{proof}
Let
$
\ell:=\lim_{t\to+\infty}t^2g(t).
$
Since $g(t)\ge0$, we have $\ell\ge0$. If $\ell>0$, then there exists
$T\ge t_0$ such that $t^2g(t)\ge \ell/2$ for all $t\ge T$. Hence
$tg(t)\ge \ell/(2t)$ for $t\ge T$, and consequently
\[
    \int_T^{+\infty}tg(t)\,dt
    \ge
    \frac{\ell}{2}\int_T^{+\infty}\frac{dt}{t}
    =
    +\infty,
\]
which contradicts $\int_{t_0}^{+\infty}tg(t)\,dt<+\infty$. Therefore
$\ell=0$, that is, $t^2g(t)\to0$. Hence $g(t)=o(t^{-2})$.
\end{proof}

The following lemma is a weighted Ces\`aro-type estimate. It can be viewed
as a continuous analogue of the classical Toeplitz theorem for summability
methods.

\begin{lemma}\label{lem:cesaro}
Let $u:[t_0,+\infty)\to[u_0,+\infty)$ be locally integrable with $t_0>0$ and
$u_0>0$. Define
\[
    R(t):=\exp\left(\int_{t_0}^{t}\frac{u(s)}{s}\,ds\right),
    \qquad t\ge t_0.
\]
Let $h:[t_0,+\infty)\to\mathbb R^m$ be locally integrable and satisfy
$\lim_{t\to+\infty}\|h(t)\|=0$. Then
\[
    \frac{1}{R(t)}
    \int_{t_0}^{t}\frac{R(s)}{s}h(s)\,ds
    \to0
    \qquad \text{as } t\to+\infty.
\]
\end{lemma}

\begin{proof}
Since $\lim_{t\to+\infty}\|h(t)\|=0$, for any $\varepsilon>0$, there exists
$T\ge t_0$ such that $\|h(t)\|\le\varepsilon$ for all $t\ge T$. Hence, for
any $t\ge T$, by the triangle inequality for vector-valued integrals,
\begin{equation}\label{eq_RR}
\begin{aligned}
\left\|
\frac{1}{R(t)}
\int_{t_0}^{t}\frac{R(s)}{s}h(s)\,ds
\right\|
&\le
\frac{1}{R(t)}
\left\|
\int_{t_0}^{T}\frac{R(s)}{s}h(s)\,ds
\right\|
+
\frac{\varepsilon}{R(t)}
\int_T^t\frac{R(s)}{s}\,ds .
\end{aligned}
\end{equation}

Since $u(t)\ge u_0>0$, we have
$
    R(t)\ge
    \exp\left(\int_{t_0}^{t}\frac{u_0}{s}\,ds\right)
    =
    \left(\frac{t}{t_0}\right)^{u_0},
$
and consequently $R(t)\to+\infty$ as $t\to+\infty$. Therefore,
\[
\lim_{t\to+\infty}
\frac{1}{R(t)}
\left\|
\int_{t_0}^{T}\frac{R(s)}{s}h(s)\,ds
\right\|
=0.
\]
Moreover, by the definition of $R(t)$, we have
$
    \frac{dR(t)}{dt}
    =
    {u(t)R(t)}/{t}
    \ge
    {u_0R(t)}/{t}.
$
Thus,
\[
    \frac{1}{R(t)}\int_T^t\frac{R(s)}{s}\,ds
    \le
    \frac{1}{u_0R(t)}\int_T^t \frac{dR(s)}{ds}\,ds
    =
    \frac{R(t)-R(T)}{u_0R(t)}
    \le
    \frac{1}{u_0}.
\]
Combining this estimate with \eqref{eq_RR}, we obtain
\[
\limsup_{t\to+\infty}
\left\|
\frac{1}{R(t)}
\int_{t_0}^{t}\frac{R(s)}{s}h(s)\,ds
\right\|
\le
\frac{\varepsilon}{u_0}.
\]
Since $\varepsilon>0$ is arbitrary, the desired conclusion follows.
\end{proof}

\begin{lemma}\label{lem:tech}
Let $\alpha>1$, $\theta\in(0,1/2)$ $t_0>0$, and let
$\delta:[t_0,+\infty)\to(0,+\infty)$ be continuously differentiable and
satisfy
\[
   \inf_{t\ge t_0}\frac{t\dot\delta(t)}{\delta(t)} >-\infty,
   \qquad
   \sup_{t\ge t_0}\frac{t\dot\delta(t)}{\delta(t)}<\frac{1}{\theta}-2.
\]
Let $r:[t_0,+\infty)\to\mathbb R^m$ be continuously differentiable and
define
\[
M(t):=
\frac{1}{t^{\alpha-1}}
\int_{t_0}^{t}s^\alpha\delta(s)
\bigl(r(s)+\theta s\dot r(s)\bigr)\,ds .
\]
If $M(t)\to0$ as $t\to+\infty$, then
$
    \lim_{t\to+\infty} \|t^2\delta(t)r(t)\|=0.
$
\end{lemma}

\begin{proof}
Set
\begin{equation}\label{eq_defy}
    y(t):=t^2\delta(t)r(t),
    \qquad
    \rho(t):=\frac{t\dot\delta(t)}{\delta(t)}.
\end{equation}
Then $r(t)=y(t)/(t^2\delta(t))$ and
\[
    \dot r(t)=\frac{\dot y(t)}{t^2\delta(t)}
    -
    \left(\frac{2}{t}+\frac{\dot\delta(t)}{\delta(t)}\right)
    \frac{y(t)}{t^2\delta(t)}.
\]
Hence
\[
r(t)+\theta t\dot r(t)
=
\frac{1-2\theta-\theta\rho(t)}{t^2\delta(t)}y(t)
+
\frac{\theta}{t\delta(t)}\dot y(t).
\]
It follows that
\[
M(t)
=
\frac{1}{t^{\alpha-1}}
\left[
\int_{t_0}^{t}s^{\alpha-2}
\bigl(1-2\theta-\theta\rho(s)\bigr)y(s)\,ds
+
\theta\int_{t_0}^{t}s^{\alpha-1}\dot y(s)\,ds
\right].
\]
Integrating by parts in the second integral gives
\begin{equation}\label{eq_My}
M(t)=
\theta y(t)
-\theta\left(\frac{t_0}{t}\right)^{\alpha-1}y(t_0)
+Q(t),
\end{equation}
where
\[
Q(t):=
\frac{1}{t^{\alpha-1}}
\int_{t_0}^{t}s^{\alpha-2}
\bigl(1-\theta(\alpha+1)-\theta\rho(s)\bigr)y(s)\,ds .
\]
Differentiating $Q(t)$ gives
\begin{equation}\label{eq_dotQ}
\dot Q(t)
=
\frac{1-\theta(\alpha+1)-\theta\rho(t)}{t}y(t)
-
\frac{\alpha-1}{t}Q(t).
\end{equation}
For simplicity, define
\begin{equation}\label{eq_pa_def}
p(t):=\frac{1}{\theta}M(t)
+\left(\frac{t_0}{t}\right)^{\alpha-1}y(t_0),
\qquad
a(t):=1-\theta(\alpha+1)-\theta\rho(t).
\end{equation}
Then \eqref{eq_My} yields
\begin{equation}\label{eq_ytt}
    y(t)=p(t)-\frac{1}{\theta}Q(t).
\end{equation}
Substituting \eqref{eq_ytt} into \eqref{eq_dotQ}, we obtain
\begin{equation}\label{eq_qdott}
    \dot Q(t)+\frac{u(t)}{t}Q(t)=\frac{a(t)}{t}p(t),
\end{equation}
where
\[
    u(t):=\frac1\theta-2-\rho(t).
\]

By the assumptions on $\rho(t)$ defined in \eqref{eq_defy}, there exists $u_0>0$ such that
$u(t)\ge u_0$ for all $t\ge t_0$. Define
$
    R(t):=\exp\left(\int_{t_0}^{t}\frac{u(s)}{s}\,ds\right).
$
Then $ \dot{R}(t)=\frac{R(t)u(t)}{t}$.
Multiplying both sides of \eqref{eq_qdott} by $R(t)$, we get
\[
    \frac{d}{dt}(R(t)Q(t))
    =
    \frac{R(t)}{t}a(t)p(t).
\]
Integrating over $[t_0,t]$ and using $Q(t_0)=0$, we obtain
\[
Q(t)=
\frac{1}{R(t)}
\int_{t_0}^{t}
\frac{R(s)}{s}a(s)p(s)\,ds .
\]

By assumption, $M(t)\to0$. Since $\alpha>1$, we also have
$\left(t_0/t\right)^{\alpha-1}y(t_0)\to0$. Hence, from
\eqref{eq_pa_def},
$
    \lim_{t\to+\infty} \|p(t)\| = 0.
$
Moreover, the assumptions on $\rho(t)$ imply that $\rho(t)$ is bounded, and
therefore $a(t)$ is bounded. Consequently,
$
    \lim_{t\to+\infty} \|a(t)p(t)\| = 0.
$
Applying Lemma~\ref{lem:cesaro} with $h(t)=a(t)p(t)$ yields
\[
    Q(t)\to 0
    \qquad \text{as } t\to+\infty.
\]
Finally, from \eqref{eq_ytt}, we obtain
$
    \lim_{t\to+\infty}\|y(t)\|=0.
$
Together with \eqref{eq_defy}, this implies the result.
\end{proof}

\bibliographystyle{siam}

\bibliography{references}

\end{document}